\pdfoutput=1
\documentclass[12pt]{article}
\usepackage{graphicx}
\usepackage{amssymb}
\usepackage{amsmath}
\usepackage{amsfonts}
\usepackage{amsthm}
\usepackage[english]{babel}
\usepackage[latin1,ansinew]{inputenc}
\usepackage{a4wide}
\usepackage{color}
\newcommand{\be}{\begin{eqnarray}}
\newcommand{\ee}{\end{eqnarray}}
\newtheorem{theo}{Theorem}%[section]
\newtheorem{numobs}{Numerical Observation}

\newcommand{\R}{\mathbb R}
\newcommand{\C}{\mathbb C}

\begin{document}

\title{Emergence of Unstable Modes\\ for Shock Waves in Ideal MHD}
\author{\it Heinrich Freist\"uhler\footnotemark[1], Felix Kleber\footnotemark[1], {\rm and} 
Johannes Schropp\footnotemark[1]}

%\author{\it Heinrich Freist\"uhler\footnotemark[1]\ ~ {\rm and}\ \  Blake Temple\footnotemark[2]}

% \large
\maketitle

\footnotetext[1]{Department of Mathematics, University of Konstanz, 78457 Konstanz, Germany}
%; Supported by DFG
%through its Excellence Grant to University of Konstanz.}
%
%\footnotetext[2]{Department of Mathematics, University of California, Davis, Davis CA, USA
%95616; Supported by NSF Applied Mathematics Grant Number DMS-070-7532.}

\begin{abstract}
This note studies classical magnetohydrodynamic shock waves in an 
inviscid fluidic plasma that is assumed to be a perfect conductor of heat as 
well as of electricity. For this mathematically prototypical material, it identifies 
a critical manifold in parameter space, across which slow classical MHD shock 
waves undergo emergence of a complex conjugate pair of unstable transverse modes.
% Two numerical schemes are described, which
%permit to (a) locate the critical manifold and (b) trace the emerging unstable
%eigenvalue. The numerics is anchored by an analytical consideration on the special case 
In the reflectionally symmetric case of parallel shocks, this emergence happens at the 
spectral value $\hat\lambda\equiv\lambda/|\omega|=0$, and the critical manifold possesses a simple explicit algebraic 
representation. 
Results of refined numerical treatment show that for only \emph{almost}\ parallel shocks the unstable
mode pair emerges from \emph{two}\ spectral values $\hat\lambda=\pm i\gamma,\gamma>0$.
\end{abstract}

%\newpage

%\setcounter{section}{-1}

\section{%{Slow and fast shock waves in ideal isothermal MHD}\label{Introduction}
%\setcounter{equation}{0}
%\subsection
%{\bf 
The equations of ideal isothermal MHD}
We consider ideal MHD in twodimensional space, 
\begin{equation}\label{claws}
\begin{aligned}
0=&\rho_t+\text{div}(\rho V)  \\
0=&(\rho V)_t+\text{div}(\rho V\otimes V+(p+\frac12|B|^2)I-B\otimes B)\\
0=&B_t+\text{div}(B\otimes V-V\otimes B).
\end{aligned}
\end{equation}
%where, with $V=(v_1,v_2), B=(b_1,b_2),$
%\begin{equation*}\label{dyades}
%V\otimes B=\begin{pmatrix}
%           v_1b_1 &v_1b_2\\ v_2b_1&v_2 b_2
%           \end{pmatrix},
%\quad\text{div}(V\otimes B)
%=
%\begin{pmatrix}
%           (v_1b_1)_x &(v_1b_2)_y\\ (v_2b_1)_x&(v_2 b_2)_y
%           \end{pmatrix}
%quad\text{et cetera.} 
%end{equation*}
The dependent variables $\rho>0,p>0,V\in\R^2$ denote the fluid's density, pressure, 
and velocity. In addition to \eqref{claws}, the magnetic field $B\in\R^2$ satisfies 
\begin{equation}\label{divzero}
\text{div}\,B=0.
\end{equation}
The fluid is assumed to be polytropic, $p=R\rho T$, and have a constant temperature $T$, so 
that $p=c^2\rho$ with constant sound speed $c$. 
By scaling, we assume without loss of generality that 
\begin{equation}\label{pisrho}
p=\rho,\quad\text{i.\ e., the speed of sound is }1.
\end{equation}
We abbreviate \eqref{claws} as
\begin{equation}\label{clawsbrief}
U_t+F(U)_x+G(U)_y=0
\end{equation}
with 
$$
U=\begin{pmatrix}
  \rho\\ \rho v_1\\ \rho v_2\\b_1\\b_2 
  \end{pmatrix},\quad
F(U)=\begin{pmatrix}
      \rho v_1\\\rho v_1v_1+p+\frac12(b_2^2-b_1^2)\\\rho v_1v_2-b_1b_2\\0\\b_2v_1-v_2b_1
     \end{pmatrix},\quad
G(U)=\begin{pmatrix}
      \rho v_2\\\rho v_2v_1-b_2b_1\\\rho v_2v_2+p+\frac12(b_1^2-b_2^2)\\b_1v_2-v_1b_2 \\0
      \end{pmatrix}.
$$
Using \eqref{divzero}, we also write it as a symmetric hyperbolic system\footnote{This issue has
an interesting history, cf.\ \cite{LaxSH,Godunov}.},
\begin{equation}\label{symmhyp}
{\bf D}(\tilde U)\tilde U_t+\tilde{\bf A}(\tilde U)\tilde U_x+\tilde{\bf B}(\tilde U)\tilde U_y=0
\end{equation}
with $\tilde U=(\rho,v_1,v_2,b_1,b_2)^\top$, ${\bf D}(\tilde U)=\text{diag}(1/\rho,\rho,\rho,1,1)$, and 
\begin{equation*}\label{symmhypmatrices}
\tilde{\bf A}(\tilde U)=\begin{pmatrix}
                         v_1/\rho&1&0&0&0\\
                         1&\rho v_1&0&0&b_2\\
                         0&0&\rho v_1&0&-b_1\\
                         0&0&0&v_1&0\\
                         0&b_2&-b_1&0&v_1
                        \end{pmatrix},\quad
\tilde{\bf B}(\tilde U)=\begin{pmatrix}
                         v_2/\rho&0&1&0&0\\
                         0&\rho v_2&0&-b_2&0\\
                         1&0&\rho v_2&b_1&0\\
                         0&-b_2&b_1&v_2&0\\
                         0&0&0&0&v_2
                        \end{pmatrix}.
\end{equation*}
Applying the chain rule, we rewrite \eqref{symmhyp} as 
\begin{equation}\label{quasilin}
 U_t+{\bf A}(U)U_x+{\bf B}(U)U_y=0, 
\end{equation}
where
\begin{equation}
%U=(\rho,\rho V,B)\quad\text{and}\quad
{\bf A}={\bf T}{\bf D}^{-1}\tilde{{\bf A}}{\bf T}^{-1},\quad 
{\bf B}={\bf T}{\bf D}^{-1}\tilde{{\bf B}}{\bf T}^{-1}
\end{equation}
with
\begin{equation}
{\bf T}=\frac{\partial U}{\partial \tilde U}
=
\begin{pmatrix}
1&0&0\\
V&\rho I_2&0\\
0&0&I_2 
\end{pmatrix}.
\end{equation}
Note that, as we have used \eqref{divzero} on the way from \eqref{claws}
to \eqref{symmhyp}, 
the matrices $\bf A$ and $\bf B$ in \eqref{quasilin} 
are \emph{not} the Jacobians of the fluxes $F$ and $G$.

\section{Slow and fast, parallel and non-parallel shock waves} 
Ideal MHD shock waves, in their prototypical form, have the structure
\begin{equation*}
U(t,x,y)=\begin{cases}
                  U^-=(\rho^-,\rho^-V^-,H^-),& (x,y)\cdot N < st,\\
                  U^+=(\rho^+,\rho^+V^+,H^+),& (x,y)\cdot N > st,
                \end{cases}
 \end{equation*}
where $N=(N_1,N_2)\in S^1$ is the direction of propagation and $s$ the speed of the shock wave.
Function \eqref{weaksolution} being a weak solution of \eqref{claws} is equivalent
to the Rankine-Hugoniot conditions
\begin{equation*}%\label{RHgeneral}
-s(U^+-U^-)+N_1(F(U^+)-F(U^-))+N_2(G(U^+)-G(U^-))=0. 
\end{equation*}
Due to rotational and Galilean invariance it is without loss of generality that 
we henceforth assume that 
$$
N=(1,0)\quad\text{and}\quad s=0;
$$ 
i.\ e. we exclusively consider shock waves of the form 
\begin{equation}\label{weaksolution}
 U(t,x,y)=\begin{cases}
                  U^-=(\rho^-,\rho^-V^-,H^-),& x < 0,\\
                  U^+=(\rho^+,\rho^+V^+,H^+),& x > 0,
                \end{cases}
\end{equation}
and the Rankine-Hugoniot conditions read
\begin{equation}\label{RH}
F(U^-)=F(U^+). 
\end{equation}
Note now first that for waves \eqref{weaksolution}, as for any solutions of \eqref{claws}
whose spatial dependence is only via $x$, the divergence-free condition \eqref{divzero}
reduces to 
\begin{equation}\label{a}
b_1=a,\quad a \text{ any constant.} 
\end{equation}
We assume \eqref{a} and simply write $b,v,w$ instead of $b_2,v_1,v_2$. 

In this paper, we are interested in Lax shocks.
Following \cite{Lax,Cabannes,FrDiss,FrRohde}, two states 
\begin{equation}\label{twostates}
U^-=(\rho^-,\rho^-v^-,\rho^-w^-,a,b^-)
\quad\text{and}\quad 
U^+=(\rho^+,\rho^+v^+,\rho^+w^+,a,b^+)
\end{equation}
that satisfy the Rankine-Hugoniot conditions \eqref{RH} constitute a
\begin{equation}\label{slowshock}
\text{slow Lax shock iff }0< \rho^+(v^+)^2<\rho^-(v^-)^2<a^2 
\end{equation}
and a 
\begin{equation}\label{fastshock}
 \text{fast Lax shock iff } a^2<\rho(v^+)^2<\rho^-(v^-)^2.
\end{equation}
Two states \eqref{twostates} do
satisfy the Rankine-Hugoniot conditions \eqref{RH} 
if and only if the two quadruples
$
(\rho^-,v^-,w^-,b^-)
$ and 
%\quad\text{and}\quad
$(\rho^+,v^+,w^+,b^+)$
have coinciding images under the mapping
\begin{equation}\label{RHmap}
\begin{pmatrix}
\rho\\ v\\ w\\ b 
\end{pmatrix}
\mapsto
\begin{pmatrix}
\rho v\\
\rho v^2+\rho+\frac12b^2\\
\rho v w-ab\\ bv-aw 
\end{pmatrix}
\end{equation}
that $F$ induces by omitting its forth, trivial component, in other words if both quadruples
satisfy the four equations
\begin{eqnarray}
\rho v&=&m\label{RHred1}\\
\rho v^2+\rho+\frac12 b^2&=&j\label{RHred2}\\
vb-aw&=&c \label{RHred3}\\
mw-ab&=&d\label{RHred4}
\end{eqnarray}
for the same values of the four parameters $m,j,c,d\in\R$.
As simple arguments\footnote{reversing, shifting, scaling, and the observation that cases with $m=0$ %and $a=0$ 
or $\rho v^2=a^2$ give no Lax shocks} 
show, we lose no generality in assuming that 
\begin{equation}
m> 0,\quad d=0,\quad\text{and } \rho v^2\neq a^2.
\end{equation}
Using \eqref{RHred4} in \eqref{RHred3} and inserting the result and \eqref{RHred1} 
in \eqref{RHred2} then yields \begin{equation}
g^{amc}(v)\equiv m\frac{1+v^2}v+\frac12\bigg(\frac{mc}{mv-a^2}\bigg)^2=j,\ \  \hbox{to be solved for } v\in
(0,a^2/m)\cup(a^2/m,\infty).
\label{geq}
\end{equation}
As for every solution $v$ of \eqref{geq}, relations \eqref{RHred1}, \eqref{RHred3},
\eqref{RHred4} provide unique associated values for $v,w$ and $b$, understanding \eqref{geq}
will give a complete picture. One distinguishes two cases.

\emph{$c = 0$: parallel shocks.} 
In this case, \eqref{geq} has two
solutions 
$$
v^\pm=\frac12\bigg(\frac jm\mp\sqrt{\big(\frac jm\big)^2-4}\bigg)\quad\text{if }\frac jm> 2,
\quad\text{with }0<v^+<v^-.
$$
The corresponding states \eqref{twostates} constitute a 
\begin{equation}\label{slowLax}
\text{slow parallel shock iff } mv^-<a^2
\end{equation}
and a
\begin{equation}
\text{fast parallel shock iff } a^2<mv^+.  
\end{equation}
The fact that the value of $a$ has no influence on the $\rho,v,w,b$ components 
of parallel shocks is easily understood by noticing that they have $b=w=0$
and thus are purely gas dynamical.

\emph{$c\neq 0$: non-parallel shocks.} In this case, $g^{amc}$ tends to $\infty$ 
not only for $v\searrow0$ und $v\nearrow\infty$, but also 
for $v \to a^2/m$. Thus for every 
$$
j>j^s_{min}(a,m,c)=\min_{(0,a^2/m)}g^{amc},
$$
\eqref{geq} has two solutions
$$
v^+_s(a,m,c,j)<v^-_s(a,m,c,j)<a^2/m
$$
that consitute a slow shock. Similarly, for every 
$$
j>j^f_{min}(a,m,c)=\min_{(a^2/m,\infty)}g^{amc},
$$
there are two solutions
$$
a^2/m<v^+_f(a,m,c,j)<v^-_f(a,m,c,j)
$$
that define a fast shock.
%\newpage

\section{Lopatinski determinant and critical manifold}  
According to Majda's theory\footnote{This is what our passing, in Sec.\ 1, through the 
symmetric hyperbolic formulation \eqref{symmhyp} is needed for.}
\cite{Majda1,Majda2} on the persistence of shock fronts,
the local-in-time stability of the planar discontinous wave \eqref{weaksolution}
is determined by the behaviour of the Lopatinski determinant
\begin{equation}
\Delta:
S_+=(\overline{\C_+}\times\R)\setminus \{(0,0)\}\to\C,\quad
 \Delta(\lambda,\omega):=\det(R^-(\lambda,\omega), J(\lambda,\omega),R^+(\lambda,\omega)),
%\quad (\lambda,\omega)\in , 
\end{equation}
where $\C_+=\{\lambda\in\C:\Re \lambda>0\}$.
While \emph{uniform stability}\ corresponds to the non-vanishing of $\Delta$ on all of $S_+$,
shocks with
\begin{equation}
\emptyset\neq\Delta^{-1}(0)\subset i\R\times\R 
\qquad\text{or}\qquad
\emptyset\neq\Delta^{-1}(0)\cap(\C_+\times \R)
\end{equation}
are \emph{neutrally stable}\ or \emph{strongly unstable}, respectively. 
The ingredients of the Lopatinski determinant are
\begin{equation*}
 \begin{aligned}
  &J(\lambda,\omega):= \lambda(U^+-U^-)+i\omega(G(U^+)-G(U^-)), \text{``jump vector''},\\
  &R^-(\lambda,\omega), \text{ base of the stable space of }
    {\bf L}^-:=(\lambda {\bf I}+i\omega {\bf B}^-)({\bf A}^-)^{-1},\\
  &R^+(\lambda,\omega), \text{ base of the unstable space of } 
       {\bf L}^+:=(\lambda {\bf I}+i\omega {\bf B}^+)({\bf A}^+)^{-1},
 \end{aligned}
\end{equation*}
where ${\bf A}^{\pm}, {\bf B}^{\pm}$ denote ${\bf A}(U^{\pm}), {\bf B}(U^{\pm})$.
The theory of hyperbolic initial-boundary value problems \cite{Kreiss, Majda1} implies that 
$R^\pm$ are well-defined bundles of constant dimension. To be precise, 
it is on 
$\C_+\times \R=S_+\setminus(i\R\times\R)$
that the 
\emph{Lopatinski matrices}\ ${\bf L}^\pm$ have constantly trivial neutral
spaces and thus ``consistent splitting'', i.\ e., stable and unstable spaces 
of constant dimensions, so that in particular
$$
d^-=\dim\text{span } R^-(\lambda,\omega)
\quad \text{and}\quad 
d^+=\dim\text{span } R^+(\lambda,\omega)
$$
are constant; for points $(\lambda,\omega)\in S_+$ with purely imaginary values of $\lambda$, 
the $R^\pm(\lambda,\omega)$ are defined as limits  from the interior of $S_+$ \cite{Kreiss}.
From the one-dimensional `Lax counting' of characteristic speeds \cite{Lax,Cabannes}, we know that 
\begin{equation}
d^-=1\quad\text{and}\quad d^+=3\quad \text{for slow MHD shocks},  
\end{equation}
while 
\begin{equation}
d^-=0\quad\text{and}\quad d^+=4\quad \text{for fast MHD shocks}.
\end{equation}
The Lopatinski determinant $\Delta$ being degree-one homogeneous in $(\lambda,\omega)$, 
we from now on fix the transverse wave number to
$$
\omega=\pm 1. 
$$
To avoid abundant notation, we also fix from now, again without loss of generality, 
\begin{equation}\label{rhois1}
 \rho^-=1
\end{equation}
and use the two parameters $\rho^+,c$ instead of the three paramters $j,m,c$. For parallel shocks, our choice
\eqref{rhois1} implies
$$
v^-=\sqrt{\rho^+},\quad v^+=1/\sqrt{\rho^+}. 
$$

In this paper we concentrate on slow shocks.\footnote{Cf.\ Trakhinin's paper 
\cite{Trakhinin} (and also \cite{Filippova,FrTrakhinin}) for other results.}
The following is a key observation.
\begin{theo}
For slow parallel MHD shocks in \eqref{claws},\eqref{divzero} with \eqref{pisrho}
and $\rho^-=1$, 
\begin{equation}
\Delta(0,\pm 1)=0 \quad\text{if}\quad {\rho^+}=\frac {a^2+2}{a^2+1}. 
\end{equation}
%For fast parallel MHD shocks,
%$
%\Delta(0,\pm 1)=0 
%$
%cannot happen.
\end{theo}
\begin{proof}
Interesting manipulations show that one can take 
$$
R^-
=
\begin{pmatrix}
 1\\
 \sqrt{\rho^+}\\
-i\sqrt{\bigg({a^2-\rho^+}\bigg)\bigg(\frac{a^2}{\rho^+}-\frac1{1-\rho^+}\bigg)}\\
a\\
0
\end{pmatrix},\quad
R^+
=
\begin{pmatrix}
                \sqrt{\rho^+}&a(\rho^+-1)&0\\
                 2&0&0\\
                 0&0&-a\sqrt{\rho^+}\\
                 0&2&0\\ 
                 0&0&1
                \end{pmatrix}.
$$
Together with $J=(0,0,i,0,0)$, this yields
$$
\Delta(0,1)=2i[\rho^+(a^2+1)-(a^2+2)].
$$
\end{proof}
\begin{figure}[ht]\center
\includegraphics[scale=0.9]{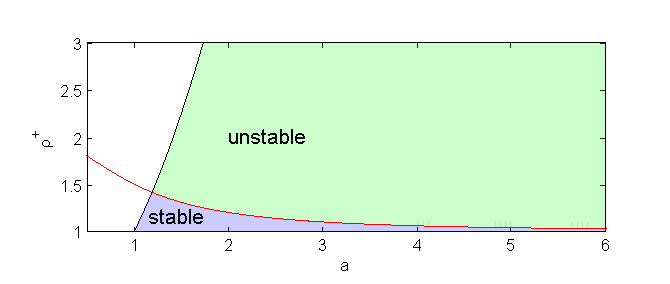}
\caption{Slow parallel shock with $\rho^-=1$, cf.\ Theorem 1. The black boundary is the Lax condition 
\eqref{slowLax}: $\rho^-(v^-)^2<a^2$.}\label{thm1}
\end{figure}

\section{Symmetry breaking}  
The situation of parallel shocks is degenerate as it possesses a reflectional
symmetry in the transverse ($y$-)direction. For the Lopatinski determinant this symmetry
means that 
$
\Delta(\lambda,-\omega)
$
vanishes exactly if 
$
\Delta(\lambda,\omega)
$
does. Perturbing the parameter $c$ away from $0$ breaks this symmetry,
and the zero of $\Delta$ that we found, for $c=0$ at $\lambda=0$, splits. 
 
For all values of $\rho^+,a,c$ that permit a (then unique) slow MHD shock wave,
we write $\Delta^{\rho^+,a,c}$ for the corresponding Lopatinski determinant. 
Starting from Theorem 1, we found the following.

\begin{numobs}
There are an $\epsilon>0$ and 
%(i) a smooth even function $a_{\text{min}}:(-\epsilon,\epsilon)$ 
%with $a_{min}(0)=2^\frac14$
%and (ii) 
two functions,
$$
R(a,c), \text{even in }c \text{ and with } R(a,0)=\frac{a^2+2}{a^2+1},
\quad\text{and}\quad
\gamma(a,c), \text{odd in }c,
$$
both defined on $\Omega_\epsilon=\{(a,c):a\ge a_{\text{min}}(c),-\epsilon<c<\epsilon\}$, 
such that 
$$
\Delta^{R(a,c),a,c}(\pm i\,\gamma(a,c),\pm 1)=0\quad\text{for all } 
c\in(-{\epsilon},{\epsilon}).
$$
\end{numobs}
\begin{figure}[ht]\center
\includegraphics[scale=0.9]{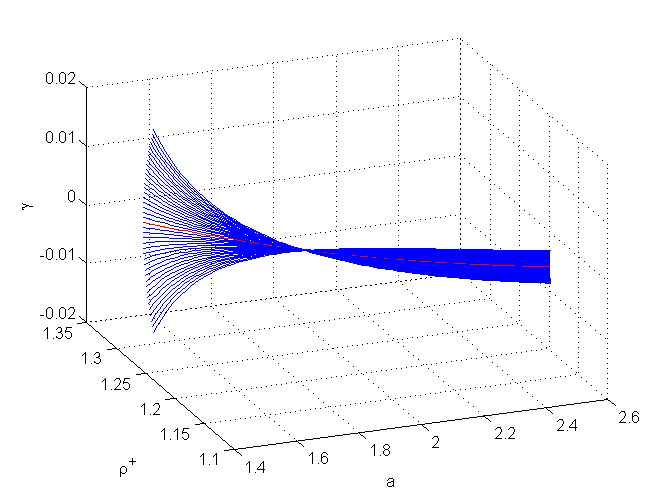}
\caption{Curves $a\mapsto (R(a,c),\gamma(a,c))$ for some values of $c$ between $-0.01$ and $0.01$.
The red curve corresponds to $c=0$ and thus to the red curve in Fig.\ 1.}\label{thm1}
\end{figure}
A detailed description of the numerics is postponed to a later publication.

\newpage
\section{Emergence of unstable modes} 
Do unstable modes emerge in families of shock waves that correspond to
parameter values which cross the critical manifold? The following is what 
we conclude from numerical computations.
\begin{numobs}
There are a $\delta>0$ and a smooth function $\alpha+i\beta:
\Omega_\epsilon\times[0,\delta)\to \C$ with
$$
\alpha(a_0,c,0)=0 \quad \text{and}\quad  \beta(a_0,c,0)=\gamma(a_0,c)
$$
and 
$$
\alpha(a_0,c,\xi)>0\quad\text{for }\xi>0
$$
such that 
$$
\Delta^{R(a_0,c),a_0+\xi,c}(\alpha(a_0,c,\xi)\pm i\,\beta(a_0,c,\xi),\pm 1)=0\quad\text{for all } 
\xi\in[0,{\delta}).
$$
\end{numobs}
This means that for $\xi>0$, 
$$
\lambda=\alpha\pm i\,\beta
$$
is an unstable eigenvalue for $\omega=\pm 1$.
\begin{figure}[ht]\center
\includegraphics[scale=0.9]{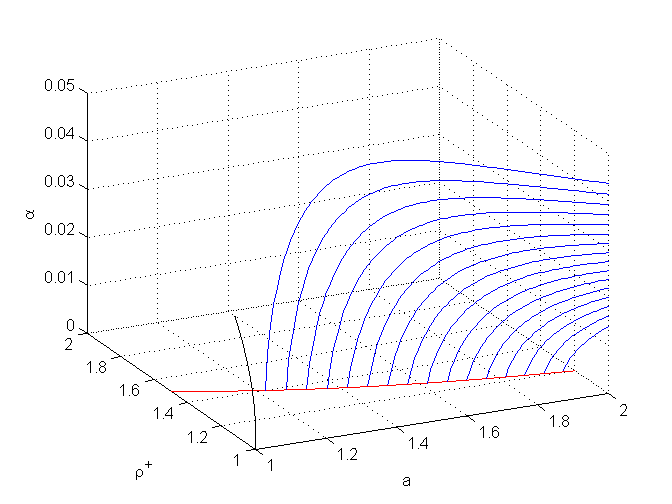}
\caption{Curves $\xi\mapsto (a_0+\xi,R(a_0,0),\alpha(a_0,0, \xi))$
for various values of $a_0$.}
%The red curve corresponds to $c=0$ and thus to the red curve in Fig.\ 1.}\label{thm1}
\end{figure}
\par\medskip
A detailed description of the numerics is again postponed to a later publication.
\newpage
{\bf Remark}. Both from a physics perspective and as the Evans function for non-ideal shock waves is 
intimately related to the Lopatinski determinant for their non-ideal counterparts \cite{ZumbrunSerre},
one expects the {\it gallopping instability}\ described in this paper to occur also in the presence
of viscosity and and electrical resistivity.

\end{document}